\documentclass[a4paper, 12pt, final]{amsart}
\usepackage{}
\usepackage{txfonts}
\usepackage[OT1]{fontenc}
\usepackage[latin1]{inputenc}
\usepackage{amssymb,latexsym}
\usepackage{type1cm,ae}
\usepackage[notref,notcite]{showkeys}
\usepackage{graphicx}
\usepackage{xspace}
\usepackage{setspace}
\usepackage[english]{babel}

\theoremstyle{plain}		
				
				\newtheorem{lemma}{Lemma}
				\newtheorem{thmA}{Theorem}
				
\theoremstyle{definition}	
\theoremstyle{remark}		\newtheorem*{remark}{Remark}

\newcommand*{\bR}{\ensuremath{\mathbb{R}}}
\newcommand*{\bN}{\ensuremath{\mathbb{N}}}

\newcommand*{\closure}[1]{\overline{#1}}
\newcommand*{\bdary}[1]{\partial #1}

\newcommand*{\Wert}{\mathord{\mbox{|\kern-1.5pt|\kern-1.5pt|}}}
\newcommand*{\ie}{\mbox{i.e.}\xspace}

\DeclareMathOperator{\dist}{dist}
\DeclareMathOperator{\diam}{diam}

\DeclareMathOperator{\modulus}{mod}

\title[mappings of finite distortion of polynomial type]{mappings of finite distortion of polynomial type}

\author{Changyu Guo}
\address[Changyu Guo]{Department of Mathematics and Statistics, University of Jyv\"askyl\"a, P.O. Box 35, FI-40014 University of Jyv\"askyl\"a, Finland}
\email{changyu.c.guo@jyu.fi}

\thanks{C.Y.Guo was partially supported by the Academy of Finland grant 120972.}

\begin{document}
\begin{abstract}
Suppose that $f: \bR^n\to\bR^n$ is a mapping of $K$-bounded $p$-mean distortion for some $p>n-1$. We prove the equivalence of the following properties of $f$: doubling condition for  $J(x,f)$ over big balls centered at origin, boundedness of multiplicity function $N(f,\bR^n)$, polynomial type of $f$ and polynomial growth condition for $f$.
\end{abstract}

\maketitle

\section{Introduction}\label{sec:first}

Throughout this paper, we call a non-constant continuous mapping $f: \Omega\to\bR^n$ a mapping of finite distortion if $f$ belongs to the Sobolev space $W_{loc}^{1,1}(\Omega,\bR^n)$ with $\frac{|Df|^n}{\log(e+|Df|)}\in L_{loc}^1(\bR^n)$
and there exists a measurable function $K:\Omega\to [1,\infty)$ so that
\begin{equation}\label{eq:definition of MFD}
|Df(x)|^n\leq K(x)J(x,f),
\end{equation}
for almost every $x\in \Omega$. Here $|Df(x)|$ and $J(x,f)$ are the operator norm and the Jacobian determinant of $Df(x)$, respectively.
When $1\leq K(x,f)\leq K<\infty$ a.e.,  $f$ is called a  $K$-quasiregular mapping.

For a mapping of finite distortion $f$, the outer distortion function $K(\cdot,f)$ and the inner distortion function $K_I(\cdot,f)$ are defined as
\begin{equation*}
K(x,f)=\frac{|Df(x)|^n}{J(x,f)} \ \ and\ \ K_I(x,f)=\frac{|D^{\#}f(x)|^n}{J(x,f)^{n-1}},
\end{equation*}
respectively, when $0<|Df(x)|, J(x,f)<\infty$, and $K(x,f)=K_I(x,f)=1$ otherwise. Here $|D^{\#} f(x)|$ is the adjoint matrix of $Df(x)$. Clearly
\begin{equation*}
K_I(x,f)^{1/(n-1)}(x,f)\leq K(x,f)\leq K_I(x,f)^{n-1},
\end{equation*}
for almost every $x\in \Omega$. 

We say that a mapping of finite distortion $f: \bR^n\to\bR^n$ has $K$-bounded $p$-mean distortion, $p\geq 1$, if there exists a constant
$K\geq 1$ such that
\begin{equation}\label{equation:bounded mean distortion 1}
\fint_{B(0,r)}K(x,f)^pdx\leq K, \ \ \forall 1\leq r<\infty.
\end{equation}

Let $f:\Omega\to\bR^n$ be a mapping of finite distortion with $K(x,f)\in L^p_{loc}(\Omega)$ for some $p>n-1$ and $E\subset\Omega$ a Borel set.
We define the counting function
\begin{equation*}
n(E,y)=\sum_{x\in f^{-1}(y)\cap E}i(x,f),
\end{equation*}
where $i(x,f)$ is the local index.

We let $A(a,r)$ be the average of $n(B(a,r),y)$ over $\bar{R}^n$ with respect to the $n$-dimensional spherical measure, \ie
\begin{equation*}
A(a,r)=\frac{2^n}{\omega_n}\int_{\bR^n}\frac{n(B(a,r),y)}{(1+|y|^2)^n}dy=\frac{2^n}{\omega_n}\int_{\closure{B}(a,r)}\frac{J_f(x)}{(1+|f(x)|^2)^n}dx.
\end{equation*}
If $a$ happens to be the origin, we simply denote it by $A(r)=A(0,r)$.

For a mapping of finite distortion $f: \bR^n\to\bR^n$, we define the lower order of $f$ to be
\begin{equation}\label{eq:order}
    \lambda_f=\liminf_{r\to\infty}\frac{\log A(r)}{\log r}.
\end{equation}
We say that $f$ has finite lower order if $\lambda_f<\infty$.

The main result of this paper is the following theorem that generalizes the corresponding statements for quasiregular mappings from~\cite{hk95}.
\begin{thmA}\label{thm:thma}
  Let $f: \bR^n\to\bR^n$ be a mapping of $K$-bounded $p$-mean distortion for some $p>n-1$. Then the following statements are equivalent
\begin{description}
\item[1] $J(x,f)$ is doubling for balls centered at the origin with radius bigger than or equal to 1.
\item[2] $f$ is of polynomial type, \ie $\lim_{x\to\infty}|f(x)|=\infty$ and so $f$ can be extended continuously to $\overline{\bR^n}$.
\item[3] $N(f,\bR^n)$ is bounded.
\item[4] $f$ has polynomial growth, \ie there exist positive constants $C_1$ and $C_2$ and integer $m$ such that $|f(x)|\leq C_1|x|^m+C_2$.
\end{description}
Moreover, each of the above condition implies that $A(r)$ is doubling for $r\geq r_0$, where $r_0>0$ is a constant. In particular, 
$f$ has finite lower order. 
\end{thmA}

The lesson we draw from Theorem~\ref{thm:thma} is that $\infty$ must be a removable singularity if $f$ does not grows too fast. Under the condition that $f$ is of polynomial type, we actually conclude that $f$ must cover the whole $\bR^n$.

\section{Preliminaries and Auxiliary results}
\subsection{Area formula}
Let $f:\Omega\to\bR^n$ be a mapping of finite distortion with distortion function $K(x,f)\in L^p_{loc}(\Omega)$ for some $p>n-1$. Then $f$ is sense-preserving, discrete and open by the result in~\cite{kkm011},
the latter meaning that $f^{-1}(y)$ cannot have accumulation points in $\Omega$. In particular, $N(y,f,A)<\infty$
whenever $A\subset\subset \Omega$. Recall that $N(y,f,\Omega)$ is defined as the number of preimages of $y$ under $f$ in $\Omega$. We also put $N(f,\Omega):=\sup_{y\in \bR^n}N(y,f,\Omega)$. 

An open, connected neighborhood $U\subset\subset\Omega$ of a point $x\in \Omega$ is called a normal neighborhood of $x$ if $f(\bdary{U})=\bdary{f(U)}$ and $U\cap f^{-1}(f(x))=\{x\}$. Denote by $U(x,f,r)$ the $x$-component of $f^{-1}(B(f(x),r))$. For the following lemma, see Chapter I Lemma 4.9 in~\cite{r93}.

\begin{lemma}[Lemma 4.9, \cite{r93}]
  For each point $x\in \Omega$ there is $\sigma_x>0$ such that $U(x,f,r)$ is a normal neighborhood of $x$ whenever $0<r\leq \sigma_x$. Moreover, $\diam U(x,f,r)\to 0$ as $r\to 0$.
\end{lemma}

Next suppose that $x\in\Omega$ and $0<r\leq \sigma_x$, where $\sigma_x$ is as in the above lemma. The local index $i(x,f)$ of $f$ at $x$ is defined as
\begin{equation*}
i(x,f)=N(f,U(x,f,r)).
\end{equation*}
Note in particular that $N(f,U(x,f,r))$ is independent of $r\leq \sigma_x$. Thus, $i(x,f)=1$ if and only if $x\in \Omega\backslash B_f$, where $B_f$ is the branch set of $f$. We refer the reader to~\cite{r93} Chapter I for this discussion.

Suppose that $f:\Omega\to\bR^n$ is a continuous, sense-preserving, discrete and open mapping. Let $\beta:[a,b)\to\bR^n$ be a path and let $x\in f^{-1}(\beta(a))$. A path $\alpha:[a,c)\to\Omega$ is called a maximal $f$-lifting of $\beta$ starting at $x$ if $\alpha(a)=x$, $f\circ\alpha=\beta|_{[a,c)}$ and if whenever $c<c'\leq b$, there does not exist a path $\alpha':[a,c)\to\Omega$ such that $\alpha=\alpha'|_{[a,c)}$ and $f\circ\alpha'=\beta|_{[a,c')}$.

Now let $x_1,\dots,x_k$ be $k$ different points of $f^{-1}(\beta(a))$ such that
\begin{equation*}
m=\sum_{j=1}^ki(x_j,f).
\end{equation*}
We say that the sequence $\alpha_1,\dots,\alpha_m$ of paths is a maximal sequence of $f$-liftings of $\beta$ starting at the points $x_1,\dots,x_k$ if each $\alpha_j$ is a maximal $f$-lifting of $\beta$ such that 
\begin{itemize}
\item $card\{j:\alpha_j(a)=x_i\}=i(x_i,f)$, $1\leq i\leq k$,
\item $card\{j:\alpha_j(t)=x\}\leq i(x,f)$ for all $x\in\Omega$ and all $t$.
\end{itemize}
The existence of maximal sequences of $f$-liftings is proved in Chapter II Theorem 3.2 in~\cite{r93}.

We close this subsection by the following well-known area formula.

\begin{lemma}\label{lemma:Area Formula}
  Let $f:\Omega\to\bR^n$ be a mapping of finite distortion with distortion function $K(x,f)\in L^p_{loc}(\Omega)$ for some $p>n-1$ and $\eta$ be a non-negative Borel-measurable function on $\bR^n$. Then
\begin{equation}\label{eq:Area Formula}
\int_{\Omega}\eta(f(x))J(x,f)dx=\int_{\bR^n}\eta(y)N(y,f,\Omega)dy.
\end{equation}
\end{lemma}
The validity of the above lemma lies in the fact that $f$ satisfies the so-called Lusin condition $N$. This is proven in~\cite{kkm01} Theorem A. Recall that a mapping $f:\Omega\to\bR^n$ is said to satisfy the Lusin condition $N$ if the implication $|E|=0\Rightarrow |f(E)|=0$ holds for all measurable sets $E\subset \Omega$. 

\subsection{Basic modulus estimates}
We first give the definitions of the $p$-modulus and weighted $p$-modulus of a path family. Let~$E$ and~$F$ be subsets of $\closure{\Omega}$. We denote by $\Gamma(E,F,\Omega)$ the path family consisting of all locally rectifiable paths joining~$E$ to~$F$ in~$\Omega$. A Borel function $\rho\colon\bR^n\to[0,\infty\mathclose]$ is said to be admissible for $\Gamma(E,F,\Omega)$ if $\int_\gamma\rho\,ds\geq1$ for all $\gamma\in\Gamma(E,F,\Omega)$. The $p$-modulus of a path family $\Gamma:=\Gamma(E,F,\Omega)$ is defined as
\begin{multline*}
  \modulus_p(\Gamma):=\inf\Big\{\int_{\Omega}\rho^p(x)\,dx :  \rho\colon\bR^n\to[0,\infty\mathclose] \text{ is an admissible }\\
  \text{Borel function for } \Gamma \Big\}.
\end{multline*}
Let $\omega:\Omega\to[0,\infty\mathclose]$ be a measurable function. The weighted $p$-modulus of the path family $\Gamma$ is then defined as
 \begin{multline*}
  \modulus_{p,\omega}(\Gamma):=\inf\Big\{\int_{\Omega}\rho^p(x)\omega(x)\,dx :  \rho\colon\bR^n\to[0,\infty\mathclose] \text{ is an admissible }\\
  \text{Borel function for } \Gamma \Big\}.
\end{multline*}
Note that when $\omega(x)=1$ for all $x\in\Omega$, we recover the usual $p$-modulus $\modulus_p$. When $p=n$, we write $\modulus_{\omega}$ instead of $\modulus_{n,\omega}$.

Next we introduce the $p$-modulus on spheres. Let $\Gamma$ be a path family in $S(a,r)$, and let $1<p<\infty$. Set
\begin{multline*}
  \modulus_p^S(\Gamma):=\inf\Big\{\int_{S(a,r)}\rho^p(x)\,dS(x) :  \rho\colon S(a,r)\to[0,\infty\mathclose] \text{ is an admissible }\\
  \text{Borel function for } \Gamma \Big\}.
\end{multline*}

The notation $\Gamma_f$ is used to denote the collection of all locally rectifiable paths in $A$ having a closed subpath on which
$f$ is not absolutely continuous.  The following $K_o$-inequality for mapping of finite distortion is due to Rajala~\cite{r04}.

\begin{lemma}[\cite{r04}, Theorem 2.1]\label{lemma:K_0 inequality}
  Let $f:\Omega\to\bR^n$ be a mapping of finite distortion with $K(x,f)\in L^{p}_{loc}(\Omega)$ for some $p>n-1$. Let $A\subset \Omega$ be
  a Borel set with
\begin{equation*}
\sup_{y\in\bR^n}N(y,f,A)<\infty,
\end{equation*}
and $\Gamma$ a family of paths in $A$. If a function $\rho$ is admissible for $f(\Gamma\backslash \Gamma_f)$, then
\begin{equation}
\modulus_{K^{-1}(\cdot,f)}(\Gamma\backslash \Gamma_f)\leq \int_{\bR^n}\rho^n(y)N(y,f,A)dy,
\end{equation}
Moreover, $\modulus_p(\Gamma_f)=0$ for all $1<p<n$.
\end{lemma}

The following important V\"as\"al\"a inequality for mappings of finite distortion was proved by Koskela and Onninen~\cite{ko06}.
Here $i(x,f)$ is the local index of $f$ at a point $x$.

\begin{lemma}[\cite{ko06}, Theorem 4.1]\label{lemma:Vaisala inequality}
  Let $f:\Omega\to\bR^n$ be a mapping of finite distortion with $K(x,f)\in L^{p}_{loc}(\Omega)$ for some $p>n-1$. Let $\Gamma$
  be a path family in $\Omega$, $\Gamma'$ a path family in $\bR^n$, and $m$ a positive integer with the following property.
  For every path $\beta:I\to\bR^n$, there are paths $\alpha_1,\cdots,\alpha_m$ in $\Gamma$ such that $f\circ\alpha_j\subset\beta$
  for all $j$ and such that for every $x\in\Omega$ and $t\in I$, $\alpha_j(t)=x$ for at most $i(x,f)$ indices j. Then
\begin{equation}
\modulus(\Gamma')\leq \frac{\modulus_{K_I(\cdot,f)}(\Gamma)}{m}.
\end{equation}
\end{lemma}

For spherical rings, we prove the following upper bound for $K_I(\cdot,f)$-modulus.

\begin{lemma}\label{lemma:upper estimate for spherical rings}
 Let $f:\Omega\to\bR^n$ be a mapping of finite distortion with $K(x,f)\in L^{p}_{loc}(\Omega)$ for some $p>n-1$.   Suppose that for some
 $a\in\Omega$ and $R>1$, $B(a,R)\subset\subset\Omega$ and that
 \begin{equation*}
    I_a^s:=\fint_{B(a,s)}K(x,f)^pdx\leq K<\infty, \ \ \forall 1\leq s\leq R.
 \end{equation*}
There exists a constant $C>0$, depending only on $n$ and $K$, such that
\begin{equation}\label{equation:upper weighted estimate}
\modulus_{K_I(\cdot,f)}(\Gamma)\leq C\log^{1-n}(R/r).
\end{equation}
for all $1\leq 2r\leq R$, where $\Gamma$ is the family of all paths connecting $\closure{B}(a,r)$ and $\bR^n\backslash B(a,R)$.
\end{lemma}
\begin{proof}
Let $\rho(x)=\frac{1}{|x-a|}\log^{-1}(R/r)$ if $x\in B(a,r)\backslash \bar{B}(a,r)$ and $\rho(x)=0$ otherwise. Then $\rho$ is an
admissible function for $\Gamma$. Let $k$ be the smallest positive integer such that $2^kr\geq R$. Then it follows that
\begin{multline*}
\modulus_{K_I(\cdot,f)}(\Gamma)\leq \int_{B(a,R)\backslash \bar{B(a,r)}}\rho^nK_I(x,f)dx\\
\leq \sum_{i=1}^k\int_{B(a,2^ir)\backslash \bar{B}(a,2^{i-1}r)}\frac{1}{|x-a|^n}\log^{-n}(R/r)K_I(x,f)dx\\
\leq C(n)\log^{-n}(R/r)\sum_{i=1}^kK=C(n)K\log^{1-n}(R/r),\\
\end{multline*}
where, in the last inequality, we have used the condition $I_a^s\leq K$ for all $1\leq s\leq R$.
\end{proof}

Next we will derive a lower bound estimate for the $K^{-1}(\cdot,f)$-modulus for spherical rings . The proof here is a small
modification to the proof of Theorem 2.2 in~\cite{r04}.

\begin{lemma}\label{lemma:lower modulus estimate}
Let $f:\Omega\to\bR^n$ be a mapping of finite distortion with $K(x,f)\in L^{p}_{loc}(\Omega)$ some $p>n-1$. Let $E$ and $F$ be two continua.
Assume that there exists a point $a$ such that $S(a,t)$ intersects both $E$ and $F$ for all $r'<t<R'$ and $B(a,R')\subset \Omega$.
Let $Y=\fint_{B(a,R')\backslash B(a,r')}K(x,f)^pdx$. Then we have the following lower bound:
\begin{equation}\label{equation:equation lower}
\modulus_{K^{-1}(\cdot,f)}(\Gamma\backslash \Gamma_f)\geq C(n)\frac{R'-r'}{R'}Y^{-\frac{1}{p}},
\end{equation}
where $\Gamma$ is the family of paths joining $E$ and $F$ in $B(a,R')\backslash \bar{B}(a,r')$.
\end{lemma}
\begin{proof}
Let $\rho$ be an admissible function for $\modulus_{K^{-1}(\cdot,f)}(\Gamma\backslash \Gamma_f)$ so that $\rho^nK^{-1}(\cdot,f)$ is integrable.
Fix $p'=n-\frac{n}{p+1}\in (n-1,n)$. Then
\begin{equation*}
\int_{S(a,t)}\rho^{p'}(x)dS(x)<\infty
\end{equation*}
for almost every $t\in (r',R')$. Now, by Lemma~\ref{lemma:K_0 inequality}, $\modulus_{p'}(\Gamma_f)=0$ and thus $\modulus_{p'}^S(\Gamma_f^t)=0$
for almost all $t\in (r',R')$, where
\begin{equation*}
\Gamma_f^t=\{\gamma\in\Gamma_f,|\gamma|\subset S(a,t)\}.
\end{equation*}
Lemma 2.4 in~\cite{r04} gives for almost all $t\in (r',R')$,
\begin{equation*}
\int_{S(a,t)}\rho^{p'}(x)dS(x)\geq \modulus_{p'}^S(\Gamma^t\backslash \Gamma_f^t)=\modulus_{p'}^S(\Gamma^t)\geq\frac{C_{p'}}{t^{p'+1-n}},
\end{equation*}
where $\Gamma^t=\{\gamma\in \Gamma:|\gamma|\subset S(a,t)\}$. We conclude that 
\begin{equation*}
1\leq Ct^{p'+1-n}\int_{S(a,t)}\rho^{p'}(x)dS(x)
\end{equation*}
for almost every $t\in (r',R')$. Integrating over $t$, we get
\begin{equation*}
(R'-r')essinf_{t\in (r',R')}\int_{S(a,t)}\rho^{p'}(x)dS(x)\leq \int_{B(a,R')\backslash B(a,r')}\rho^{p'}(x)dx.
\end{equation*}
Write $E=B(a,R')\backslash B(a,r')$. Then $|E|$ is bounded from above and below by dimensional constants times $(R'-r')R'^{n-1}$.
Combining the above two inequalities gives
\begin{equation}\label{equation:equation 1}
1\leq CR'\big(|E|^{-1}\int_E\rho^{p'}(x)dx\big)^{1/p'}.
\end{equation}
Writing $\rho^{p'}(x)=\rho^{p'}(x)K^{-p'/n}(x,f)K^{p'/n}(x,f)$ and using Holder's inequality, we see the righthand side
of~\eqref{equation:equation 1} is at most
\begin{equation}\label{equation:equation 2}
CR'\big(|E|^{-1}\int_E \rho^n(x)K^{-1}(x,f)dx\big)^{1/n}\big(|E|^{-1}\int_E K^p(x,f)dx\big)^{1/np}.
\end{equation}
Combining estimate~\eqref{equation:equation 1} with~\eqref{equation:equation 2}, we obtain the desired estimate
\begin{equation*}
    C(n)\frac{R'-r'}{R'}Y^{-\frac{1}{p}}\leq \int_E \rho^n(x)K^{-1}(x,f)dx.
\end{equation*}
\end{proof}

\begin{lemma}\label{lemma:upper modulus estimate}
Let $f: \bR^n\to\bR^n$ be a mapping of $K$-bounded $p$-mean distortion for some $p>n-1$. Then $f(\bR^n)$ cannot omit a set of positive $n$-capacity.
\end{lemma}
\begin{proof}
Note that in our situation, all the requirements of Corollary 1.3 in~\cite{ko06} are satisfied and thus the conclusion follows.
\end{proof}

\subsection{Growth estimate}
In this subsection, we will derive a growth estimate on $f$ under the assumption that $J(x,f)$ is doubling for balls centered at the origin with radius bigger than or equal to 1. The idea of the proof is from Lemma 12.3 in~\cite{or09}. For the statement of the following results, we introduce the notations  $L(0,R)=\sup\{|f(x)-f(y)|:|x-y|\leq R\}$ and $l(0,R)=\inf\{|f(x)-f(y)|:|x-y|\geq R\}$.

\begin{lemma}\label{lemma:distortion estimate}
 Let $f: \bR^n\to\bR^n$ be a non-constant mapping of $K$-bounded $p$-mean distortion for some $p>n-1$.
 If $J(x,f)$ is $C$-doubling for all balls centered at the origin with radius bigger than or equal to 1, then there exist $C'=C'(n,K)>0$ and $R_0>0$ such that
\begin{equation}
L(0,R)\leq C'R^k
\end{equation}
for all $R>R_0$, where $k=\frac{\log_2C+1}{n}$.
\end{lemma}
\begin{proof}
By Lemma~\ref{lemma:distortion estimate}, we may set $R_0$ to be such a number that $L(0,R)=1$.  Let $R>R_0$. We fix a point $a\in\closure{B}(0,R)$ such that
\begin{equation}
|f(a)-f(0)|=L(0,R).
\end{equation}
By Lemma~\ref{lemma:upper modulus estimate}, we may choose  $\gamma:[0,\infty)\to \bR^n\backslash B(f(0),L(0,R))$,
starting at $f(a)$ with $|\gamma|$ unbounded and a maximal $f$-lifting of $\gamma'$ of $\gamma$ starting at $a$.
Then $|\gamma'|$ is also unbounded. Lemma~\ref{lemma:lower modulus estimate} implies that
\begin{equation}\label{equation:lower1}
\modulus_{K^{-1}(\cdot,f)}(\Gamma\backslash \Gamma_f)\geq C(n,K)R^{-1},
\end{equation}
where $\Gamma$ is the family of curves joining $\closure{B}(0,1)$ to $|\gamma'|\cap\closure{B}(2R)$ in $B(0,2R)$.
Since every path $\eta\in f(\Gamma)$ intersects $f(\closure{B}(0,1))$ and $\bR^n\backslash B(f(0),L(0,R))$,
the function $\rho:\bR^n\to[0,\infty]$ defined by
\begin{equation*}
\rho(y)=\frac{2}{L(0,R)}\chi_{f(B(0,2R))}(y)
\end{equation*}
is a test function for $f(\Gamma)$. By Lemma~\ref{lemma:Area Formula},
\begin{equation}\label{equation:upper}
\int_{\bR^n}N(y,f,B(0,2R))\rho^n(y)dy=\frac{2^n}{L(0,R)^n}\int_{B(0,2R)}J(x,f)dx\leq C_1\frac{R^s}{L(0,R)^n}
\end{equation}
for $s=\log_2C$. The claim now follows from Lemma~\ref{lemma:K_0 inequality}, (\ref{equation:lower1}) and (\ref{equation:upper}).
\end{proof}
 
\textbf{Corollary} Let $f: \bR^n\to\bR^n$ be a non-constant mapping of $K$-bounded $p$-mean distortion for some $p>n-1$. If $A(r)$ is $C$-doubling for $r\geq 1$ and $L(0,r)\leq C_1 l(0,r)$ for $r>r_0$, then there exist $C'=C'(n,K)>0$ and $R_0'>0$ such that 
\begin{equation}
l(0,2R)\geq C'R^\alpha
\end{equation}
for all $R>R_0$, where $\alpha=\frac{1}{n(\log_2C+1)}$. In particular, $f(\bR^n)$ contains balls of arbitrary large radius.
\begin{proof}
The proof is quite similar to Lemma~\ref{lemma:distortion estimate}. Let $R_0$ be the same number as in the proof of Lemma~\ref{lemma:distortion estimate} and $R_0'=\max\{r_0,R_0\}$. For $R>R_0'$ we write 
\begin{equation*}
\rho(y)=\frac{2}{l(0,2R)}\chi_{f(B(0,2R))}(y)
\end{equation*}
to be the test function for $f(\Gamma)$ and
\begin{multline*}
\int_{\bR^n}N(y,f,B(0,2R))\rho^n(y)dy=\frac{2^n}{l(0,2R)^n}\int_{B(0,2R)}J(x,f)dx\\
\leq \frac{2^n(1+L(0,2R)^2)^n}{l(0,2R)^n}\int_{B(0,2R)}\frac{J(x,f)}{(1+|f(x)|^2)^n}dx\\
\leq C(n)l(0,2R)^nR^s\\
\end{multline*}
for $s=\log_2C$. 
\end{proof}
\remark This result can be viewed as a covering theorem for mappings of finite distortion, comparing it with Corollary 1.2 in~\cite{r07}.

\section{Proof of theorem A}
\begin{proof}
The implication ``$\textbf{3}\Rightarrow \textbf{1}"$ follows from modulus estimates. First of all, Lemma~\ref{lemma:Area Formula} implies that
\begin{equation*}
\int_{2B}J(x,f)dx=\int_{f(2B)}N(y,f,2B)dy\leq N(f,\bR^n)|f(2B)|
\end{equation*}
and
\begin{equation*}
\int_{B}J(x,f)dx=\int_{f(B)}N(y,f,B)dy\geq |f(B)|.
\end{equation*}
Therefore it suffices to show that
\begin{equation*}
|f(2B)|\leq C|f(B)|
\end{equation*}
whenever $B\subset \bR^n$ is a ball centered at origin with radius bigger than or equal to $1$.

To this end, fix a ball $B=B(0,\eta)\subset \bR^n$ with $\eta\geq 1$. Since $f$ is open, $f(0)$ is an interior point of the open connected
set $U=f(B)$. Let $d=\dist(f(0), \bdary{U})$ and set
\begin{equation*}
r=\max_{y\in\bdary{U'}}|f(0)-y|,
\end{equation*}
where $U'=f(2B)$. Then clearly $2r\geq \diam(U')$. Let $L_0$ be the line segment of length $d$ joining $f(0)$ to $\bdary{U}$ and let $L_1$
be a half line in $\bR^n$ joining a point in $\bdary{B(f(0),r)}\cap \bdary{U'}$ to $\infty$ in $\bR^n\cup\{\infty\}\backslash B(f(0),r)$.
 Then some lifts of $L_0$ and $L_1$ join $0$ to $\bdary B$, and $\bdary 2B$ to $\infty$, respectively, and hence by
 Lemma~\ref{lemma:lower modulus estimate}
\begin{equation*}
\modulus_{K^{-1}(\cdot,f)}(\Gamma\backslash \Gamma_f)\geq C>0,
\end{equation*}
where $\Gamma$ is the family of paths joining $f^{-1}(L_0)$ and $f^{-1}(L_1)$ in $3B$. On the other hand, Lemma~\ref{lemma:K_0 inequality} gives us
\begin{equation*}
\modulus_{K^{-1}(\cdot,f)}(\Gamma\backslash \Gamma_f)\leq N(y,f,\bR^n)\modulus(L_0,L_1,\bR^n)\leq C(\log\frac{r}{d})^{1-n}.
\end{equation*}
We thus obtain $r\leq Cd$, which implies that $f(B)$ and $f(2B)$ have comparable volumes, as desired.

Next, we prove the equivalence of $\textbf{2}$ and $\textbf{3}$. First we show that $\textbf{2}$ implies $\textbf{3}$. To this end,
fix a point $y\in f(\bR^n)$. Since $f$ is discrete and $\lim_{x\in\infty}|f(x)|=\infty$, there exists a ball $B=B(0,r)\subset \bR^n$ such that
\begin{equation}\label{equation:multiplicity inequality}
N(y,f,\bR^n)=N(y,f,B)\leq \sum_{x\in f^{-1}(y)}i(x,f)=:M<\infty.
\end{equation}
Now suppose that there exists a point $v\in X$ with
\begin{equation*}
\sum_{x\in f^{-1}(v)}i(x,f)>M,
\end{equation*}
and choose a compact path $\gamma$ beginning at $v$ and ending at $y$. Then there are at least $M+1$ lifts $\gamma^j$ of $\gamma$
starting at $f^{-1}(v)$, and each of them either ends at some $x\in f^{-1}(y)$ or leaves every compact subset of $\bR^n$.
The latter cannot happen for any $j$ since $f$ is of polynomial type. But,  by~\eqref{equation:multiplicity inequality}, the former
can occur for at most $M$ of the lifts $\gamma^j$. This is a contradiction and thus $\textbf{2}$ implies $\textbf{3}$.

Now assume $\textbf{3}$ and suppose that $\textbf{2}$ does not hold. Then there exists a sequence $(a_i)$ of points in $\bR^n$,
such that $|a_i|$ increases to infinity but
\begin{equation}\label{equation:polynomial}
\limsup_{i\to\infty}|f(a_i)-f(0)|=R<\infty.
\end{equation}
Passing to a subsequence if necessary, we may assume that
\begin{equation*}
\diam f(B(0,|a_i|))\geq R/2
\end{equation*}
for every $i\in\bN$. Note that $f$ is continuous and open, and hence it is monotone. It follows that
\begin{equation}\label{equation:lower}
\diam f(S(0,|a_i|))\geq R/2.
\end{equation}

Since $N(f,\bR^n)<\infty$, we may fix a $\delta>0$ so that $U(x_j,\delta)$ is a normal neighborhood of $x_j$ for every $x_j\in f^{-1}(f(0))$. Then
\begin{equation}\label{equation:covering}
\cup_{x_j\in f^{-1}(f(0))}U(x_j, \delta)\subset B(0,t)
\end{equation}
for some $t>0$. When $|a_i|>t$, let $\Gamma_i$ be the family of all paths joining $B(f(0),\delta)$ and $f(S(0,|a_i|))$ in $\bR^n$. Then
by~\eqref{equation:polynomial}, \eqref{equation:lower} and the $n$-Loewner property of $\bR^n$, there exists a constant $C>0$ such that
$\modulus(\Gamma_i)\geq C$ for every $i$. Denote by $\Gamma_i'$ the family of all lifts $\gamma'$ of $\gamma\in \Gamma_i$ starting at $S(0,|a_i|)$.

By~\eqref{equation:covering}, every $\gamma'\in \Gamma'_i$ either intersects $B(0,t)$ or leaves every compact set in $\bR^n$. The latter
family of paths have $K_I(\cdot,f)$-modulus zero. All other paths start at $S(0,|a_i|)$ and intersect $S(0,t)$, so
\begin{equation*}
\modulus_{K_I(\cdot,f)}(\Gamma'_i) \leq C(n,K,t)\log^{1-n}(|a_i|/t)\to 0, \ \ as\ i\to\infty
\end{equation*}
by the remark after Lemma~\ref{lemma:upper estimate for spherical rings}. We thus get a contradiction to Lemma~\ref{lemma:Vaisala inequality}.
Therefore, $\textbf{2}$ follows from $\textbf{3}$.

Now we show the implication ``$\textbf{1}\Rightarrow \textbf{3}"$. First we suppose $R$ is large and  fix a point $y\in f(B(0,R))$ with
$m$ preimage points $x_1,\cdots,x_m$ inside $B(0,R)\subset\bR^n$, where $m$ is a large positive integer to be estimated later. We choose
$\delta>0$ so that $U(x_i,\delta)\subset B(0,R)$ is a normal neighborhood for each $i=1,\cdots,m$, and the sets $U(x_i,\delta)$ are pairwise disjoint.

By Lemma~\ref{lemma:upper modulus estimate}, we can choose a point $f(q)\in f(\bR^n)$ with $|y-f(q)|$ as large as desired,  and a path
\begin{equation*}
\gamma:[0,\infty)\to \bR^n\backslash B(y,|y-f(q)|),
\end{equation*}
starting at $f(q)$, such that $|\gamma|$ is unbounded. Then
\begin{equation}\label{equation:lower2}
\modulus(\Gamma)=\modulus(\closure{B}(y,\delta),|\gamma|,\bR^n)\geq C_1(\log\frac{|y-f(q)|}{\delta})^{1-n},
\end{equation}
where $C_1>0$ does not depend on $q$.

By Lemma~\ref{lemma:distortion estimate} and our choice of $\gamma$, there exist $\alpha>0$ such that
\begin{equation}\label{eq:upper ineq}
\dist(0,f^{-1}(|\gamma|))\geq C|y-f(q)|^\alpha
\end{equation}
when the right-hand term is larger than some fixed constant. For each $\eta\in\Gamma$, there are at least $m$ maximal $f$-liftings $\eta_i$
starting at the points $x_i\in B(0,R)$. Moreover, by the above estimate, each of them intersects $\bR^n\backslash B(0,C|y-f(q)|^\alpha)$.
Denote the family of all such lifts by $\Gamma'$. Then by Lemma~\ref{lemma:Vaisala inequality} and Lemma~\ref{lemma:upper estimate for spherical rings},
\begin{equation}\label{equation:upper2}
\modulus(\Gamma)\leq\frac{\modulus_{K_I(\cdot,f)}(\Gamma')}{m}\leq \frac{C_1}{m}\log^{1-n}(C_2|y-f(q)|^\alpha/R).
\end{equation}
Combining $(\ref{equation:lower2})$ and $(\ref{equation:upper2})$ yields
\begin{equation}\label{eq:compare order}
|y-f(q)|^\alpha\leq \frac{R}{\delta}|y-f(q)|^{\frac{C}{m^{1/(n-1)}}}.
\end{equation}
We conclude from~\eqref{eq:compare order} that $m\leq(\frac{C}{\alpha})^{n-1}$. To see this, suppose that the conclusion is not true,
 \ie $m>(\frac{C}{\alpha})^{n-1}$. We may let $|y-f(q)|$  be as 
large as we wish and obtain a contradiction with~\eqref{eq:compare order}.

The implication ``$\textbf{1}\Rightarrow \textbf{4}"$ follows immediately from Lemma~\ref{lemma:distortion estimate}. The proof of implication
``$\textbf{4}\Rightarrow \textbf{3}"$ is the same as that in the implication ``$\textbf{1}\Rightarrow \textbf{3}"$. Indeed, \eqref{eq:upper ineq} 
follows immediately from Lemma~\ref{lemma:distortion estimate}.

Finally we only need to show that one of the above four conditions implies that $A(r)$ is doubling for $r\geq r_0$ for some positive constant $r_0$. 
We want to prove that $\textbf{3}$ implies that $A(r)$ is doubling for all $r\geq r_0$. The proof is similar to the proof for the implication 
 ``$\textbf{3}\Rightarrow \textbf{1}"$.
Let $B=B(0,r)$ with $r\geq r_0$, where $r_0>0$ is to be determined later. Lemma~\ref{lemma:Area Formula} gives us
\begin{multline*}
A(2r)=\int_{2B}\frac{J(x,f)}{(1+|f(x)|^2)^n}dx\\
=\int_{f(2B)}\frac{N(y,f,2B)}{(1+|y|^2)^n}dy
\leq N(f,\bR^n)\int_{f(2B)}\frac{1}{(1+|y|^2)^n}dy
\end{multline*}
and
\begin{equation*}
A(r)=\int_{B}\frac{J(x,f)}{(1+|f(x)|^2)^n}dx=\int_{f(B)}\frac{N(y,f,2B)}{(1+|y|^2)^n}dy\geq \int_{f(B)}\frac{1}{(1+|y|^2)^n}dy.
\end{equation*}
Therefore it suffices to show that
\begin{equation*}
\int_{f(2B)}\frac{1}{(1+|y|^2)^n}dy\leq C\int_{f(B)}\frac{1}{(1+|y|^2)^n}dy.
\end{equation*}

Then we follow the proof of implication of ``$\textbf{3}\Rightarrow \textbf{1}"$.  Since $f$ is open, $f(0)$ is an interior point of the open connected
set $U=f(B)$. Let $d=\dist(f(0), \bdary{U})$ and set
\begin{equation*}
r'=\max_{y\in\bdary{U'}}|f(0)-y|,
\end{equation*}
where $U'=f(2B)$. Now we choose $r_0$ so that $r\geq r_0$ implies that $d\geq 2$. Arguing as before, we conclude that $r'\leq Cd$. 
Thus we have
\begin{equation*}
\int_{f(B)}\frac{1}{(1+|y|^2)^n}dy\geq \int_{B(f(0),d)}\frac{1}{(1+|y|^2)^n}dy\geq C(n)\int_0^d\frac{t^{n-1}}{(1+t^2)^n}dt.
\end{equation*}
On the other hand,
\begin{equation*}
\int_{f(2B)}\frac{1}{(1+|y|^2)^n}dy\leq \int_{B(f(0),r')}\frac{1}{(1+|y|^2)^n}dy\leq C(n)\int_0^{r'}\frac{t^{n-1}}{(1+t^2)^n}dt.
\end{equation*}
Note that if we split the integral $\int_0^{l}\frac{t^{n-1}}{(1+t^2)^n}dt$ into two parts,
\begin{equation*}
\int_0^{l}\frac{t^{n-1}}{(1+t^2)^n}dt=\int_0^1\frac{t^{n-1}}{(1+t^2)^n}dt+\int_1^l\frac{t^{n-1}}{(1+t^2)^n}dt,
\end{equation*}
then, for the second integral, we have that
\begin{equation*}
\frac{1}{n\cdot2^n}(1-\frac{1}{l^n})\leq \int_1^l\frac{t^{n-1}}{(1+t^2)^n}dt\leq \frac{1}{n}(1-\frac{1}{l^n}).
\end{equation*}
Since $1-\frac{1}{d^n}\geq 1-\frac{1}{2^n}$ and $1-\frac{1}{r^n}\leq 1$, we conclude that $A(r)$ is doubling for $r\geq r_0$.
\end{proof}

\remark Note that the boundedness of $N(f,\bR^n)$ does not imply that $J(x,f)$ being doubling for all balls centered at the origin as one may consider the
mapping 
\begin{equation*}
  F(x)=
  \begin{cases}
   \frac{x}{|x|}\log^{-1}\frac{1}{|x|} & \text{if } |x| \leq 1/e \\
    x & \text{if } |x| \geq 1/e .
  \end{cases}
\end{equation*}

\textbf{Acknowledgements}

We thank Academic Professor Pekka Koskela for many useful suggestions and for carefully reading the
manuscript. We also thank Academy Research Fellow Kai Rajala for many valuable discussions.

\end{document}